\documentclass[11pt,a4paper,reqno]{amsart}
\usepackage{calc,amsfonts,amsthm,amscd,epsfig,psfrag,amsmath,amssymb,enumerate,graphicx}
\usepackage{dsfont}

\reversemarginpar
\newlength\fullwidth
\numberwithin{equation}{section}

\DeclareMathSymbol{\leqslant}{\mathalpha}{AMSa}{"36} 
\DeclareMathSymbol{\geqslant}{\mathalpha}{AMSa}{"3E} 
\DeclareMathSymbol{\eset}{\mathalpha}{AMSb}{"3F}     
\renewcommand{\leq}{\;\leqslant\;}                   

\def\1{\ifmmode {1\hskip -3pt \rm{I}} \else {\hbox {$1\hskip -3pt \rm{I}$}}\fi}

\newtheorem{Theorem}{Theorem}[section]
\newtheorem{Lemma}[Theorem]{Lemma}
\newtheorem{Proposition}[Theorem]{Proposition}

\newtheorem{remark}[Theorem]{Remark}

\newcommand{\cA}{\ensuremath{\mathcal A}}

\newcommand{\cD}{\ensuremath{\mathcal D}}

\newcommand{\cG}{\ensuremath{\mathcal G}}

\newcommand{\cI}{\ensuremath{\mathcal I}}

\newcommand{\cL}{\ensuremath{\mathcal L}}

\newcommand{\cN}{\ensuremath{\mathcal N}}

\newcommand{\bbE}{{\ensuremath{\mathbb E}} }

\newcommand{\bbL}{{\ensuremath{\mathbb L}} }

\newcommand{\bbN}{{\ensuremath{\mathbb N}} }

\newcommand{\bbP}{{\ensuremath{\mathbb P}} }

\newcommand{\bbR}{{\ensuremath{\mathbb R}} }

\newcommand{\bbZ}{{\ensuremath{\mathbb Z}} }

\newcommand{\ent}{{\rm Ent} }
\newcommand{\var}{\operatorname{Var}}

  \let\h=\eta      
      \let\o=\omega

     \let\L=\Lambda 
\let\O=\Omega      

\def\\{\hfill\break}

\def\thsp{\thinspace}

\def\tthsp{\kern .083333 em}

\def\?{\mskip -10mu}
\def\indbox#1{\hbox to \parindent{\hfil\ #1\hfil} }

\newfam\msafam
\newfam\msbfam
\newfam\eufmfam
\def\hexnumber#1{%
  \ifcase#1 0\or 1\or 2\or 3\or 4\or 5\or 6\or 7\or 8\or
  9\or A\or B\or C\or D\or E\or F\fi}
\font\tenmsa=msam10
\font\sevenmsa=msam7
\font\fivemsa=msam5
\textfont\msafam=\tenmsa
\scriptfont\msafam=\sevenmsa
\scriptscriptfont\msafam=\fivemsa
\edef\msafamhexnumber{\hexnumber\msafam}%
\mathchardef\restriction"1\msafamhexnumber16
\mathchardef\ssim"0218
\mathchardef\square"0\msafamhexnumber03
\mathchardef\eqd"3\msafamhexnumber2C
\def\QED{\ifhmode\unskip\nobreak\fi\quad
  \ifmmode\square\else$\square$\fi}
\font\tenmsb=msbm10
\font\sevenmsb=msbm7
\font\fivemsb=msbm5
\textfont\msbfam=\tenmsb
\scriptfont\msbfam=\sevenmsb
\scriptscriptfont\msbfam=\fivemsb

\font\teneufm=eufm10
\font\seveneufm=eufm7
\font\fiveeufm=eufm5
\textfont\eufmfam=\teneufm
\scriptfont\eufmfam=\seveneufm
\scriptscriptfont\eufmfam=\fiveeufm

\def\({\left(}
\def\){\right)}
\let\neper=e
\let\ii=i

\def\ie{\hbox{\it i.e.\ }}

\let\sset=\subset

\def\nep#1{ \neper^{#1}}

\def\tc{\thsp | \thsp}
\let\<=\langle
\let\>=\rangle

\def\Var{ \mathop{\rm Var}\nolimits }

\def\gap{\mathop{\rm gap}\nolimits}

\outer\def\nproclaim#1 [#2]#3. #4\par{\medbreak \noindent
   \talato(#2){\bf #1 \Thm[#2]#3.\enspace }%
   {\sl #4\par }\ifdim \lastskip <\medskipamount
   \removelastskip \penalty 55\medskip \fi}
\def\thmm[#1]{#1}
\def\teo[#1]{#1}
\def\sttilde#1{%
\dimen2=\fontdimen5\textfont0
\setbox0=\hbox{$\mathchar"7E$}
\setbox1=\hbox{$\scriptstyle #1$}
\dimen0=\wd0
\dimen1=\wd1
\advance\dimen1 by -\dimen0
\divide\dimen1 by 2
\vbox{\offinterlineskip%
   \moveright\dimen1 \box0 \kern - \dimen2\box1}
}

\begin{document}
\title[FA1f out of equilibrium]{Fredrickson-Andersen one spin facilitated model out of equilibrium}

\author[O. Blondel]{O. Blondel}
\email{oriane.blondel@ens.fr}
\address{Univ. Paris Diderot, Sorbonne Paris Cit\'e, LPMA, UMR 7599,
F-75205 Paris, France}

\author[N. Cancrini]{N. Cancrini}
\email{nicoletta.cancrini@roma1.infn.it}
\address{Dip. Matematica Univ. L'Aquila, I-67010 L'Aquila, Italy}

\author[F. Martinelli]{F. Martinelli}
\email{martin@mat.uniroma3.it}
\address{Dip. Matematica, Univ. Roma Tre, Largo S.L.Murialdo 00146, Roma, Italy}

\author[C. Roberto]{C. Roberto}
\email{croberto@math.cnrs.fr}
\address{MODAL'X, Universit\'e Paris Ouest Nanterre La D\'efense, 200 avenu de la R\'epublique 92000 Nanterre, France}

\author[C. Toninelli]{C. Toninelli}
\email{Cristina.Toninelli@lpt.ens.fr}
\address{Laboratoire de Probabilit\'es et Mod\`eles Al\`eatoires
  CNRS-UMR 7599 Universit\'es Paris VI-VII 4, Place Jussieu F-75252
  Paris Cedex 05 France}


\thanks{Work supported by the European Research Council through the ``Advanced Grant''  PTRELSS 228032 and the French Ministry of Education through ANR-2010-BLAN-0108}

\begin{abstract}
We consider the Fredrickson and Andersen one spin facili\-ta\-ted mo\-del (FA1f)
on an infinite connected graph with polynomial growth. 
Each site with rate one refreshes its occupation variable to
a filled or to an empty state with probability $p\in[0,1]$ or $q=1-p$
respectively, provided that at least one of its nearest neighbours is
empty. We study the non-equilibrium dynamics started from an initial
distribution $\nu$ different from the stationary product $p$-Bernoulli
measure $\mu$. We assume that, under $\nu$,  the mean distance between
two nearest empty sites is uniformly bounded. We then prove convergence to equilibrium when the vacancy
density $q$ is above a proper threshold $\bar q<1$. The convergence is exponential or stretched exponential, depending on the growth of the graph.
In particular it is exponential on $\bbZ^d$ for $d=1$ and stretched exponential for $d>1$. Our result can be generalized to other {\it non cooperative} models.

\end{abstract}

\maketitle


\section{Introduction}

Fredrickson-Andersen one spin facilitated model (FA1f) \cite{FA1,FA2} belongs to the class of interacting particle systems known as
Kinetically Constrained Spin Models (KCSM),
 which have been introduced and very much studied  in the physics
literature  to model liquid/glass transition and more
generally glassy dynamics (see \cite{RS,GST} and references therein).
A configuration for a KCSM is given by
assigning to each vertex $x$ of a (finite or infinite) connected graph
$\cG$ its occupation variable
$\eta_x \in\{0,1\}$, which corresponds to an empty or filled site
respectively.  The evolution is given by  Markovian stochastic dynamics
of Glauber type.  With rate one each site refreshes its occupation variable to
a filled or to an empty state with probability $p\in[0,1]$ or $q=1-p$
respectively, provided that the current configuration satisfies an
a priori specified local constraint.  
For FA1f the constraint  at $x$ requires at least one of its nearest
neighbours to be empty.\footnote{For FA1f a single empty site is sufficient to
  ensure irreducibility of the chain. KCSM in which a fine subset of
  empty sites is able to move around and empty the whole space are called
 \emph{non-cooperative} and are in general easier to analyze than
 cooperative ones.}
Note that  (and this is a general feature of KCSM) the constraint which should be satisfied to allow creation/annihilation of a particle at $x$ does not
involve $\eta_x$. Thus FA1f dynamics satisfies detailed balance w.r.t.\ the Bernoulli product
measure at density $p$, which is therefore an invariant reversible measure for the process. 
Key features of FA1f model and more generally of KCSM
are that a completely filled configuration is blocked (for generic KCSM other blocked configurations may occur) - namely
 all creation/destruction rates are identically equal to
zero in this configuration -, and that due to the
constraints the dynamics is not attractive, so that monotonicity
arguments valid for {\it e.g.}\ ferromagnetic stochastic Ising models cannot be
applied.
Due to the above properties the basic issues concerning the large time
behavior of the process are non-trivial. 

In \cite{CMRT} it has been proved that the model on $\cG=\mathbb Z^d$ is ergodic for any $q>0$ with a positive spectral gap  which shrinks to zero as $q\to 0$ corresponding to the occurrence of diverging mixing times. 
A key issue both from the mathematical and the physical point of view is
what happens when the evolution does not start from the equilibrium measure
$\mu$. The analysis of this setting usually requires much more detailed
information than just the positivity of the spectral gap,
{\it e.g.}\ boundedness of the logarithmic Sobolev constant or positivity of the entropy
constant uniformly in the system size. The latter requirement
certainly does not hold  (see Section 7.1 of \cite{CMRT}) and even the basic question of whether
convergence to $\mu$ occurs remains open in the infinite volume case. Of course, due to the
existence of blocked configurations, convergence to $\mu$ cannot hold
\emph{uniformly} in the initial configuration and one could try to prove it a.e.\ or
in mean w.r.t.\ a proper initial distribution $\nu\neq \mu$.  

From the point of
view of physicists, a particularly relevant case (see
{\it e.g.}\   \cite{LMSBG})  
is when $\nu$ is a product Bernoulli($p'$) measure with $p'\neq p$
and $p'\neq 1$). 
In this case the most natural guess is
that convergence to equilibrium occurs for any local (\ie depending
on finitely many occupation variables) function $f$ \ie
\begin{equation}
  \label{eq:basic}
\lim_{t\to \infty}\int d\nu(\h) \bbE_\h\bigl(f(\h_t)\bigr)=\mu(f)  
\end{equation}
where $\h_t$ denotes the process started from $\h$ at time $t$ and that
the limit is attained exponentially fast.

The only other case of KCSM where this result has been proved
\cite{CMST} (see also \cite{FMRT}) is the East model, that is a one dimensional model in
which the constraint at $x$ requires the neighbour to the right of $x$
to be empty. The strategy used to prove convergence to equilibrium for East model in \cite{CMST} relies however heavily on the oriented character of the East constraint and cannot be extended to FA1f model. We also recall that in \cite{CMST} a perturbative result has been established proving exponential convergence for any
one dimensional KCSM with finite range jump rates and positive spectral gap  (thus including FA1f at any $q>0$), provided
 the initial distribution $\nu$ is ``not too far'' from
the reversible one ({\it e.g.}\ for $\nu$ Bernoulli at density $p'\sim p$). 
\smallskip

Here we prove convergence to equilibrium for
FA1f on a infinite connected graph $\cG$ with polynomial growth (see
the definition in sec.\ \ref{setting} below) when the equilibrium
vacancy density $q$ is above a proper threshold $ \bar q$ (with $\bar
q<1$) and the starting measure $\nu$ is such that the mean distance
between two nearest empty sites is uniformly bounded.  That includes
in particular any non-trivial Bernoulli product measure with $p'\neq
p$ but also the case in which $\nu$ is the Dirac measure on a fixed
configuration with infinitely many empty sites and such that the distance between two nearest empty sites is uniformly bounded. The derived convergence
is either exponential or stretched exponential depending on the growth of the graph.
In the particular case $\cG=\bbZ^d$, we can prove exponential
relaxation only for $d=1$. If $d>1$ we get  a stretched exponential behavior.  
Although our result can be generalized to other \emph{non cooperative
  models} (see {\it e.g.}\ \cite{CMRT} for the definition of this class ), to let the
paper be more readable we consider here only the FA1f case. 



\smallskip
\noindent
We finish with a short road map of the paper. In section \ref{general models} we introduce the notations and give the main result.
The main strategy is described in section \ref{generale} but it can be
summarized as follows.
We first replace $\bbE_\nu(f(\eta_t))$ with a similar quantity but
computed w.r.t the FA1f finite volume process (actually a finite state,
continuous time  Markov
chain evolving in a finite ball of radius proportional to time $t$
around the support of $f$). This first reduction is standard and it
follows easily from the so-called {\it finite speed of
  propagation}. Then we show that, with high probability, only the
evolution of the chain inside a suitable ergodic component
matters. The ergodic component is chosen in such a way that the log-Sobolev constant for the restricted chain is much smaller
than $t$. This second reduction is new and it is at this stage that
the restriction on $q$ appears and that all the difficulties of the
non-equilibrium dynamics appear.  
Its implementation requires the estimate of the spectral gap of the
process restricted to the ergodic component (see section
\ref{ergodicgap}) and the study of the persistence of zeros out of
equilibrium (see section \ref{perszeros}). Finally, in section \ref{proofmt}  we prove the
main result of the paper.


\section{Notation and Result} \label{general models}


\subsection{The graph} \label{setting}

Let $\cG=(V,E)$ be an infinite, connected graph with vertex set  $V$, edge set $E$ and graph distance
$d(\cdot,\cdot)$.
Given $x\in V$ the set of neighbors of $x$ 
will be denoted by $\cN_x$. For all $\Lambda\sset V$ we call $\mathrm{diam}(\Lambda)=\sup_{x,y\in\Lambda}d(x,y)$
the diameter of $\Lambda$ and $\partial \Lambda = \{x \in V \setminus \Lambda \colon d(x, \Lambda)=1\}$ its (outer) boundary. 
Given a vertex $x$ and an integer $r$, $B(x,r)=\{y \in V \colon d(x,y)\le r\}$ denotes the ball centered at $x$ and of radius $r$.
We introduce the growth function $F \colon \mathbb{N}\setminus\{0\} \to \mathbb{N} \cup \{ \infty\}$ defined by
$$
F(r)=\sup_{x \in V} |B(x,r)|
$$
where $|\cdot |$ denotes the cardinality. 
Then we say that $\cG$ has $(k,D)$-{\it polynomial growth} if $F(r)\le k\,r^D$ for all $r \ge 1$, with 
$k$ and $D$ two positive constants. An example of such a graph is given by the $d$-dimentional square lattice
$\bbZ^d$ that has $(3^d,d)$-polynomial growth (with the constant $3^d$ certainly not optimal).

 
\subsection{The probability space}

The configuration space is $\Omega=\{0,1\}^{V}$ equip\-ped with the Bernoulli product measure $\mu$ of parameter $p$. Similarly we define $\Omega_\Lambda$ and $\mu_\Lambda$ for any subset
$\Lambda\subset V$. Elements of $\Omega$ ($\Omega_\Lambda$) will be denoted by Greek letters $\eta,\,\omega,\,\sigma$ ($\eta_\Lambda,\,\omega_\Lambda,\,\sigma_\Lambda$) etc.
Furthermore, we introduce the
shorthand notation $\mu(f)$ to denote the expected value of $f$ and $\Var(f)$ for its variance (when it exists).


\subsection{The Markov process} \label{Markov process}

The interacting particle model that will be studied here is a 
Glauber type Markov process in $\O$, reversible w.r.t.\ the measure
$\mu$.
It  can be informally described
as follows. Each vertex $x$ waits an independent mean one exponential
time and then, provided that the current configuration $\sigma$ is such that one of the neighbors of $x$ 
(i.e. one site $y \in \cN_x$) is empty,
the value $\sigma(x)$ is refreshed with a
new value in $\{0,1\}$ sampled from a Bernoulli $p$ measure and the whole procedure starts again.

The generator $\cL$ of the process can be constructed in a standard way (see {\it e.g.}\ \cite{Liggett}). It acts on local functions as
\begin{equation}\label{gen}
\cL f(\sigma)=\sum_{x\in V}c_x(\sigma)[q\sigma(x)+p(1-\sigma(x))][f(\sigma^x)-f(\sigma)]
\end{equation}
where $c_x(\sigma)=1$ if $\prod_{y\in \cN_x}\sigma(y)=0$ and
$c_x(\sigma)=0$ otherwise (namely the constraint requires at least one empty neighbor),
$\sigma^x$ is the configuration $\sigma$ flipped at site $x$, $q\in[0,1]$ and $p=1-q$.
It is a non-positive self-adjoint operator on $\bbL^2(\O,\mu)$ with domain $Dom(\cL)$, core 
$\cD(\cL)=\{f\colon\Omega\to\bbR\,\;{\rm s. t.}\,
\sum_{x\in V}\sup_{\sigma\in\Omega}|f(\sigma^x)-f(\sigma)|<\infty\}$
and Dirichlet form given by
$$
\cD(f)=\sum_{x\in V}\mu\(c_{x} \Var_x(f)\),\quad f\in Dom(\cL) .
$$
Here $\Var_{x}(f)\equiv\int d\mu(\o(x)) f^2(\o) - \(\int d\mu(\o(x))f(\o)\)^2$ denotes the local variance with respect to
the variable $\o(x)$ computed while the other variables are held fixed.
To the generator $\cL$ we can associate the Markov semigroup
$P_t:=\nep{t\cL}$ with reversible invariant measure $\mu$.
We denote by $\sigma_t$ the process at time $t$ starting from the configuration $\sigma$.
Also, we denote by $\mathbb E_{\eta}(f(\eta_t))$  the expectation over the process generated by $\cL$ at time $t$ and started at configuration $\eta$ at time zero and, with a slight abuse of notation, we let
$$\mathbb E_{\nu}(f(\sigma_t)):=\int d\nu(\eta)\mathbb E_{\eta}(f(\eta_t))$$
and let $\mathbb P_{\nu}$ be the distribution of the process started with distribution $\nu$ at time zero.

For any subset $\Lambda\subset V$ and any configuration $\eta\in\Omega$

\begin{equation}\label{genl}
\cL_\Lambda^\eta f(\sigma)=\sum_{x\in\Lambda}c_{x,\Lambda}^\eta(\sigma)[q\sigma(x)+p(1-\sigma(x))][f(\sigma^x)-f(\sigma)]
\end{equation}
where $c_{x,\Lambda}^\eta(\sigma)=c_x(\sigma_\Lambda\eta_{\Lambda^c})$ where $\sigma_\Lambda\eta_{\Lambda^c}$ is the configuration equal to $\sigma$ on $\Lambda$ and equal to $\eta$ on $\Lambda^c$. When $\eta$ is the empty  configuration we write simply $c_{x,\Lambda}$ and $\cL_\Lambda$.


\subsection{Main Result}

In order to state our main theorem, we need some notations. For any vertex $x \in V$, and any configuration 
$\sigma \in \Omega$, let
$$
\xi^x(\sigma) = \min_{y \in V :\ \sigma(y)=0} \{d(x,y)\} 
$$
be the distance of $x$ from the set of empty sites of $\sigma$.

\begin{Theorem} \label{main}
Let $q >1/2$. Assume that the graph $\cG$ has $(k,D)$-polynomial growth and $f \colon \Omega \to \bbR$
is a local function with $\mu(f)=0$.
Let $\nu$ be a probability measure on $\Omega$ such that 
$\kappa:=\sup_{x \in V} \bbE_\nu (\theta_o^{\xi^x})< \infty$ for some $\theta_o >1$. Then,  there exists a positive constant $c=c(q,k, D,\kappa,|\mathrm{supp(f)}|)$ such that 
$$
|\mathbb E_{\nu}(f(\sigma_t))| \le c \vert| f \vert|_\infty
\begin{cases}
e^{-t/c} & \mbox{if } D=1 \\
e^{-[t/(c\log t)]^{1/D}} & \mbox{if } D >1 .
\end{cases}
\qquad \forall t \ge 2 .
$$
\end{Theorem}

\begin{remark}
We expect that our results hold also for $0<q\le \frac 12$. This needs a more precise control of
the behavior of $\xi_t^x=\xi^x(\sigma_t)$. In dimension one we can obtain a better
threshold by calculating further time derivatives of $u(t)=\bbE_\eta(\theta^{\xi_t})$, see
Proposition \ref{prop:xi} below.
\end{remark}

\begin{remark}
Observe that if $\nu$ is a Dirac mass on some configuration $\eta$,  the condition reads $\sup_{x \in V} \theta_o^{\xi^x(\eta)}< \infty$. This encodes the fact that $\eta$
has infinitely many empty sites and that, in addition, the distance between two nearest empty sites is uniformly bounded. This condition is different from the case of the East model in \cite{CMST} where the condition on the initial configuration was the presence of an infinite number of zeros.
\end{remark}

\begin{remark}
If one considers the case in which $\nu$ is the product of
Bernoulli-$p'$ on $\cG$, one has that, for all $\theta < 1/p'$
and all $x \in \cG$,
$$
\bbE_\nu (\theta^{\xi^x}) = \sum_{k=0}^\infty \theta^k \bbP_\nu(\xi^x = k)
\le \sum_{k=0}^\infty \theta^k {(p')}^{|B(x,k)|} \le\sum_{k=0}^\infty (\theta p')^k =\frac{1}{1-p'\theta} .
$$
Hence, $\kappa \le \frac{1}{1-p'\theta_o}$ for $\theta_o \in (1, 1/p')$. 
In particular Theorem \ref{main} applies to any initial probability measure, product of Bernoulli-$p'$ on $\cG$,
with $p' \in [0,1)$.
\end{remark}

\begin{remark}
Note that graphs with polynomial growth are amenable. We stress anyway that there exist amenable graphs 
which do not satisfy our assumption. This is due 
to Proposition \ref{inizio} below that gives a useless bound in the case of
amenable graphs with intermediate growth
(i.e. faster than any polynomial but slower than any exponential, see \cite{grigorchuk}). 
The same happens to any graph with exponential growth (such as for example 
any regular $n$-ary tree ($n \ge 2$)).
\end{remark}


\section{From infinite to finite volume}

This section provides a general result that will be the starting point of our analysis. 
The strategy developped here (and given in Section \ref{generale} below in a general setting) might be of indepedendent interest.
The idea is first to reduce the study of the evolution of the process from infinite to a finite ball of radius proportional to $t$.
Then to small sets on some ergodic component so that the log-Sobolev constant is much smaller than $t$.

The first reduction is standard and known as {\sl the finite speed of propagation}. Namely, given a local function 
$f$ with $\mathrm{supp(f)} \subset B(x,r)$ for some $x \in V$, and some integer $r$, 
 we have (see {\it e.g.}\ \cite{SFlour}) for any initial measure $\nu$ on $\Omega$
$$
|\bbE_\nu (f(\sigma_t) - f(\sigma_t^\Lambda))| \le c \| f\|_\infty e^{-t} 
$$
where $\sigma_t^\Lambda$ is the configuration at time $t$ of the process starting from $\sigma_{\Lambda}$,
on the finite volume $\Lambda = B(x,r+ 100 t)$ with empty boundary condition and $c$ is some positive constant 
depending on $|\mathrm{supp(f)}|$.
Hence,
\begin{equation}\label{finitespeed}
|\bbE_\nu (f(\sigma_t))| \le  |\bbE_\nu(f(\sigma_t^\Lambda))| + c\| f\|_\infty e^{-t} .
\end{equation}

Next we  divide $\Lambda$ into $n$ connected subsets $\Lambda_1$, $\Lambda_2, \dots, \Lambda_n$ such that $\cup_i \Lambda_i=\Lambda$ and $\Lambda_i \cap \Lambda_j =\emptyset$ for all $i \neq j$. Such a decomposition will be called a partition of $\Lambda$.

Given such a partition of $\Lambda$, let $\cA$ be the set of configurations containing at least two empty sites in each $\Lambda_i$. Namely,
\begin{equation} \label{defA}
\cA = \bigcap_{i=1}^n \{ \sigma \in \Omega_\Lambda \mbox{ s.t. } \sum_{x \in \Lambda_i} (1-\sigma(x)) \geq 2\} .
\end{equation}
With these notations we can now state our result.

\begin{Proposition} \label{inizio}
Fix $\Lambda \subset V$ and $f$ local, with  $\mathrm{supp(f)} \subset \Lambda$ and $\mu(f)=0$.
Then, there exists a constant $c=c(q,|\mathrm{supp(f)}|)$ such that
for any partition $\Lambda_1$, $\Lambda_2, \dots, \Lambda_n$ of $\Lambda$, for any initial probability measure $\nu$ on $\Omega$,
it holds, 
\begin{align*}
|\bbE_\nu(f(\sigma_t^\Lambda))| & \le c \vert|f\vert|_\infty \left( n e^{-qm}  
+ t\, \sup_{s \in [0,t]} \bbP_\nu(\eta_s^\Lambda \notin \cA)  + |\Lambda| e^{-t/3} \right. \\
& \qquad  \qquad \qquad \left.  + \exp\left\{-\frac{t}{c} +c|\Lambda| e^{-t/(cM)} \right\}\right)
\end{align*}
where $m:=\min\{|\Lambda_1|,\dots,|\Lambda_n|\}$ and $M:=\max\{|\Lambda_1|,\dots,|\Lambda_n|\}$, provided
that $ne^{-qm}<1/2$.
\end{Proposition}

In order to prove Proposition \ref{inizio} we need first a general result on Markov processes.

\subsection{Preliminary results on Markov processes} \label{generale}
We here give a general result which links the behavior of a Markov process on a finite space to that of a restricted Markov process. 
We use it to reduce the evolution of the FA1f process to small sets on some ergodic component. In this section $S$ is a finite space.
Recall  that a transition rate matrix $Q=(q(x,y))_{x,y\in S}$ is such that for any $x,y\in S$
$$
q(x,y)\ge 0\quad\text{for} \, x\neq y\quad\text{and } \sum_{y\in S} q(x,y)=0.
$$
and $Q$ univocally defines a continuous time Markov chain $(X_t)_{t \ge 0}$ as follows \cite{Liggett}. If $X_t=x$, then the
process stays at $x$ for an exponential time with parameter $c(x)=-q(x,x)$. At the end of that time, it jumps to $y\neq x$
with probability $p(x,y)=q(x,y)/c(x)$, stays there for an exponential time with parameter $c(y)$, etc. 
 Fix $\cA\subset S$
and set $\hat \cA= \cA\cup\{y\notin \cA\colon q(x,y)>0\, \text{for some } x\in
\cA\}$. Let $(\hat X_t)_{t \ge 0}$ be a continuous time Markov chain with transition
rate matrix $\hat Q=(\hat q(x,y))_{x,y\in \hat \cA}$ such that $\forall \,x\in \cA$ and $\forall y\in\hat\cA$, 
$\hat q(x,y)= q(x,y)$. 
Assume that $(X_t)_{t \ge 0}$  and $(\hat X_t)_{t \ge 0}$ are reversible with respect to some probability measures $\pi$ and $\hat\pi$ respectively. Then, we define the spectral gap $\hat \gamma$ of the hat chain as
$$
\hat \gamma := \inf_{f \colon f \neq \mathrm{const}} \frac{\sum_{x,y}\hat \pi(x)p(x,y)(f(y)-f(x))^2}{2\var_{\hat \pi}(f)}
$$
and the log-Sobolev constant $\hat \alpha$ as
$$
\hat \alpha := \sup_{f  \colon f \neq \mathrm{const}} \frac{2\ent_{\hat \pi}(f^2)}{\sum_{x,y}\hat \pi(x)p(x,y)(f(y)-f(x))^2}
$$
where $\ent_{\hat \pi}(f)=\hat \pi(f \log f) - \hat \pi(f) \log \hat \pi (f)$ denotes the entropy of $f$.
See \cite{bibbia} for an introduction of these notions.

\begin{Proposition}\label{p:simpatica}
Let $(X_t)_{t \ge 0}$, $(\hat X_t)_{t \ge 0}$,  $\pi$, $\hat\pi$,  $\hat \gamma$ and $\hat \alpha$ as above. 
Then,  for  all initial probability measure $\nu$ on $S$ and all $f\colon S\to \bbR$ with $\pi(f)=0$, it holds
\begin{align} \label{paperino}
\vert\bbE_\nu(f(X_t))\vert\le \vert\hat\pi(f)\vert +\vert|f\vert|_\infty\left(4\bbP_\nu(\cA_t^c)+\exp\left\{-\hat\gamma \frac t2+e^{-\frac{2t}{\hat\alpha}}\log \frac{1}{\hat \pi^*} \right\}\right)
\end{align} 
where $\cA_t=\{X_s\in \cA,\, \,\,\forall s\le t\}$ and $\hat \pi^*:=\min_{x \in S} \hat \pi(x)$.
\end{Proposition}

\begin{remark}
The usual argument (see \cite{HS,SZ}) using the log-Sobolev  constant would lead to
$$
\vert\bbE_\nu(f(X_t))\vert\le
\vert| f \vert|_\infty
\exp\left\{- \gamma \frac t2+\exp\{-\frac{2t}{\alpha}\}\log \frac{1}{\pi^*} \right\} .
$$
On finite  subsets $\L$ of $\bbZ^d$ (with $\pi=\mu$) the log-Sobolev
constant grows proportionally to the volume $|\L|$ (see \cite{CMRT}). In view
of the finite speed property (see \eqref{finitespeed}), one has to consider
$\Lambda = B(x,r+100t)$
so that the log-Sobolev constant grows as $t^d$. 
Hence, using $\log \pi^* \simeq t^d$, one would get the useless bound
$$
\vert\bbE_\nu(f(X_t))\vert\le
\vert| f \vert|_\infty
\exp\left\{- \gamma \frac t2+ct^d \right\} .
$$
The main improvement in proposition \ref{p:simpatica} comes from the fact that we deal with a restricted (the hat) chain for which the log-Sobolev constant $\hat \alpha$ is much smaller than $\alpha$, in particular much smaller than $t$ so that
the dominant term in $\exp\left\{-\hat\gamma \frac t2+\exp\{-\frac{2t}{\hat\alpha}\}\log \frac{1}{\hat \pi^*} \right\}$
is given by the gap term $\hat \gamma t$. The price to pay are the extra terms in \eqref{paperino}
that one has to analyze separately.
\end{remark}

\begin{proof}
Fix a probability measure $\nu$ and a function $f$ with $\pi(f)=0$ and let $g=f-\hat\pi(f)$.
Then 
\begin{equation}\label{eq:g}
\vert\bbE_\nu(f(X_t))\vert\le\vert\hat\pi(f)\vert +\vert|g\vert|_\infty \bbP_\nu(\cA_t^c)+\vert\bbE_\nu(g(X_t)\mathds{1}_{\cA_t})\vert .
\end{equation}
We now concentrate on the last term in \eqref{eq:g}. By definition of the chains $(X_t)_{t \ge 0}$ and  $(\hat X_t)_{t \ge 0}$
one has  
$$\bbE_\nu(g(X_t)\mathds{1}_{\cA_t})
=\int_\cA d\nu(x)\bbE_x(g(\hat X_t)\mathds{1}_{\{\hat X_s\in \cA,\,\forall s\le t\}}) .
$$ 
Hence, by H\"older inequality, we have
\begin{align*}
\vert\bbE_\nu(g(X_t)\mathds{1}_{\cA_t})\vert &=\vert\int_\cA d\nu(x)\bbE_x(g(\hat X_t)(1-\mathds{1}_{\{\hat X_s\in \cA,\,\forall s\le t\}^c})\vert\\
&\le \vert\hat{\pi}(h\,\hat{P}_t g(\hat X_t))\vert+2\vert|f\vert|_\infty\bbP_\nu(\cA_t^c)\\ &\le 
\vert|h\vert|_{L^\beta(\hat\pi)} \vert|g(\hat X_t)\vert|_{L^{\beta'}(\hat\pi)} +2\vert|f\vert|_\infty\bbP_\nu(\cA_t^c)
\end{align*}
where $h=d\nu/d\hat\pi$ and $\beta$, $\beta' \ge 1$, that will be chosen later, are such that $1/\beta+1/\beta'=1$. To bound the previous expression
take $\beta'=1+e^{\frac{2t}{\hat\alpha}}$. Using the hypercontractivity property \cite{gross} (see {\it e.g.}\ \cite[chapter 2]{bibbia})
and the spectral gap we obtain 
\begin{equation*}\label{eq:bibbia}
 \vert|g(\hat X_t)\vert|_{L^{\beta'}(\hat\pi)} \le\vert|g(\hat X_{\frac t2})\vert|_{L^2(\hat\pi)}\le \,e^{-\hat\gamma\frac t2}
\vert|g\vert|_{L^2(\hat\pi)}\le e^{-\hat\gamma\frac t2}\vert|f\vert|_\infty.
\end{equation*}
On the other hand
\begin{equation*}\label{eq:h}
\vert|h\vert|_{L^{\beta}(\hat\pi)} \le \left(\int h d\hat\pi\right)^{\frac 1\beta}\vert|h\vert|_\infty^\frac{\beta-1}{\beta}=\vert|h\vert|_\infty^{\frac {1}{\beta'}}\le \exp\{e^{-\frac{2t}{\hat\alpha}}\log \vert|h\vert|_\infty\}
\end{equation*}
and the proof is completed since $\vert|h\vert|_\infty \le \log \frac{1}{\hat \pi^*}$.
\end{proof}

\subsection{Proof of Proposition \ref{inizio}}
This section is dedicated to the proof of Proposition \ref{inizio}.

In the sequel $c$ will denote a constant depending on $q$ and
$|\mathrm{supp(f)}|$ and whose value may change from line to line.

Our aim is to apply Proposition \ref{p:simpatica}. Let us define the setting. First $S=\Omega_\Lambda$.
Define $\mathcal{A}$ as the set of configurations in $\Omega_\Lambda$ such that there exist at least 
two empty sites in each set $\Lambda_i$ (see \eqref{defA}), and $\mathcal{A}_t=\{ \sigma_s^\Lambda \in \mathcal{A} \mbox{ for all } s \le t \}$.
Also, let $\hat{\mathcal{A}} = \{\sigma \in \Omega_\Lambda \colon 
 \; \sum_{x \in \Lambda_i} (1-\sigma(x)) \ge 1, \; \mbox{ for all } i=1,\dots, n\} $,
 i.e. the set of configurations that can be obtained from $\mathcal{A}$ by a legal flip for the process.
Let $\omega$ be the entirely filled configuration (i.e. such that $\omega(x)=1$ for all $x \in V$).
The transition rates are,  $\forall \sigma,\, \eta\in\Omega_\Lambda$,
 $$
 q(\sigma,\eta)=\begin{cases} c_{x,\Lambda}(\sigma) [q\sigma(x)+p(1-\sigma(x))]\,&{\rm if} \,\eta=\sigma^x\\
 0\,&{\rm otherwise}\end{cases} 
 $$
and 
$$
\hat q(\sigma,\eta)=\begin{cases}
c_{x,\Lambda}^\omega(\sigma) [q\sigma(x)+p(1-\sigma(x))]\,& \mbox{if } \sigma \in \cA \mbox{ and } \eta=\sigma^x \in \cA \\
0&{\rm otherwise.}
\end{cases}
$$
The process $(X_t)_{t \ge 0}$ is $(\sigma_t^\Lambda)_{t \ge 0}$.
The process $(\hat X_t)_{t \ge 0}$ is the process $(\sigma_t^\Lambda)_{t \ge 0}$ started from $\sigma \in  \cA$ and killed on $\hat{\mathcal{A}}^c$.
Then $\pi=\mu_\Lambda$ and 
$\hat \pi (\cdot) = \mu_\Lambda(\cdot \tc \hat\cA) $.
Thus, thanks to Proposition \ref{p:simpatica}, we have
\begin{align*}
|\bbE_\nu(f(\sigma_t^\Lambda))|\le 
\vert\hat\pi(f)\vert +\vert|f\vert|_\infty\left(2\bbP_\nu(\cA_t^c)+\exp\left\{-\hat\gamma \frac t2+e^{-\frac{2t}{\hat\alpha}}\log \frac{1}{\hat \pi^*}\right\}\right) .
\end{align*} 
We now study each term of the last inequality separately. 

If we recall that $\mu_\Lambda(f)=\mu(f)=0$ and using a union bound, we have
\begin{equation}\label{eq:fame1}
\vert\hat\pi(f)\vert=\frac{\vert\mu_\Lambda(f(1-\mathds{1}_{\hat\cA^c}))\vert}{\mu_\Lambda(\hat\cA)}\le\vert|f\vert|_\infty\frac{\mu_\Lambda(\hat\cA^c)}{\mu_\Lambda(\hat\cA)}\le\vert| f\vert|_\infty\,\frac{ne^{-qm}}{1-ne^{-qm}} .
\end{equation}
We now deal with the term $\bbP_\nu(\cA_t^c)$.

Let $\cI_t$ be the event that there exists a site in $\Lambda$ with more than $2t$ rings in the time interval $[0,t]$. Then, by standard large deviation of Poisson variables 
and a union bound, there exists a universal positive constant $d$ such that $\bbP_\nu(\cA_t^c \cap \cI_t) \le d |\Lambda| e^{-t/3}$. 
Furthermore, using a union bound on all the rings on the event $\cI_t^c$, we have
$$
\bbP_\nu(\cA_t^c \cap \cI_t^c) \le 2 t\sup_{s \in [0,t]} \bbP_\nu(\sigma_s^\Lambda \notin \cA) .
$$
We deduce that
$$
\bbP_\nu(\cA_t^c) \le c\left( t\sup_{s \in [0,t]} \bbP_\nu(\sigma_s^\Lambda \notin \cA) +  |\Lambda| e^{-t/3} \right) .
$$

Next we analyse the log-Sobolev constant $\hat \alpha$ and the spectral gap constant $\hat \gamma$. For that purpose, let us introduce a new process $\tilde{X}$ with transition rates : 

$$
\tilde q(\sigma,\eta)=\begin{cases}
c_{x,\Lambda_i}^\omega(\sigma) [q\sigma(x)+p(1-\sigma(x))]
\,& \text{if $x\in \L_i$ and $\sigma,\, \sigma^x \in \hat \cA$}\\
0&{\rm otherwise.}
\end{cases}
$$
Clearly $\hat{\alpha}\le \tilde{\alpha}<\hat \gamma$. Moreover $\hat{\pi}$ is a reversible measure for the tilde process.
Observe that
$\hat \cA = \bigcap_{i=1}^n \hat \Omega_i$
where
$$
\hat \Omega_i=\{ \sigma \in \Omega_{\Lambda} \mbox{ s.t. } \exists x_i \in \Lambda_i \mbox{ with } \sigma(x_i)= 0 \}.
$$
If we denote by $\hat \mu_i(\cdot)=\mu_{\Lambda_i}(\cdot \tc \hat \Omega_i)$, the product structure of
$\mu$ and  $\cA$ transfers to $\hat \pi$ so that
$\hat \pi = \otimes_i \hat \mu_i$. Furthermore, $\tilde X_t$ restricted to each $\hat \Omega_i$ is ergodic and reversible with respect to
$\hat \mu_i$. Hence we can define the associated spectral gap $\tilde{ \gamma_i}$ and log-Sobolev constant
$\tilde{\alpha_i}$. By the well-known tensorisation property of the Poincar\'e and the log-Sobolev inequalities (see {\it e.g.}\ \cite[Chapter 1]{bibbia}),
we conclude that  $\tilde{ \gamma}=\min(\tilde \gamma_1,\dots,\tilde \gamma_n)$ and $\tilde \alpha=\max(\tilde \alpha_1,\dots,\tilde \alpha_n)$.
Then, Proposition \ref{gap} below shows that $\tilde \gamma \ge c$ and
$\tilde \alpha_i \leq c|\Lambda_i|$.
Hence,
$$
\exp\left\{-\hat\gamma \frac t2+\exp\{-\frac{2t}{\hat\alpha}\}\log \frac{1}{\hat \pi^*} \right\}
\leq 
\exp\left\{-\frac{t}{c}+ c|\Lambda|e^{-t/(cM)} \right\} .
$$
This ends the proof.

\begin{Proposition}[\cite{CMRTpraga}] \label{gap}
Let $A \subset V$ be connected, $\hat \Omega_A = \{\sigma \in \Omega \colon 
\sum_{x \in A}(1-\sigma(x))) \ge 1 \}$ and $\hat \mu_A (\cdot) = \mu_A (\cdot \tc \hat \Omega_A)$. 
Let $\omega$ be the entirely filled configuration (i.e. such that $\omega(x)=1$ for all $x \in V$) and
for $x \in A$, define
$\hat c_x (\sigma) := c_{x,A}^\omega(\sigma) \mathds{1}_{\sigma^x \in \hat \Omega_A}$ on
$\hat \Omega_A$.
Then, there exists a constant $c=c(q)$ such that  
$$
\hat \gamma_A := \inf_{f: f \neq const.} \frac{\sum_{x \in A} \hat \mu_A (\hat c_x \var_x(f))}{\var_{\hat \mu_A} (f)} \ge c
$$
and
$$
\hat \alpha_A := \sup_{f: f \neq const.} \frac{\ent_{\hat \mu_A} (f)} {\sum_{x \in A} \hat \mu_A (\hat c_x \var_x(f))} \le c |A| .
$$
\end{Proposition}

\begin{proof}
The first part on the spectral gap is proven in \cite[Theorem 6.4 page 336]{CMRTpraga}.
In section \ref{ergodicgap} we give an alternative proof which gives a better bound
for small $q$ and can be extended to non cooperative models different from FA1f.

The second part easily follows from the standard bound \cite{diaconis-saloffcoste,Saloff}
$$
\hat \alpha_A \le \hat \gamma_A^{-1} \log \frac{1}{\hat \mu_A^*} 
$$
where
$\hat \mu_A^*:=\min_{\sigma \in \hat \Omega_A} \hat \mu_A(\sigma)
\ge \exp\{-c |A|\}$. 
\end{proof}


\section{Persistence of zeros out of equilibrium}\label{perszeros}

In this section we study the behavior of the minimal distance from a fixed site at which one finds a vacancy. The result that we obtain will be used in Section \ref{proofmt} for the proof of our main Theorem \ref{main}. Indeed Proposition \ref{prop:xi} will be a key ingredient in order to estimate the second term in the inequality of Proposition \ref{inizio}, namely the probability the process gets out of the component $\cA$ of the configuration set which requires two empty sites on each small volume $\Lambda_i$. 
For any $\sigma\in\{0,1\}^{V}$ and any
$x\in V$ define $\xi^x(\sigma)$ as the minimal distance at which one finds an empty site starting from $x$,
$$
\xi^{x}(\sigma)=\min_{y\in V\colon \sigma(y)=0}\{d(x,y)\}
$$
with the
convention that $\min \emptyset=+\infty$, ($\xi^x(\sigma)=0$ if $\sigma(x)=0$).  

\begin{Proposition} \label{prop:xi}
Consider the FA1f process on a finite set $\Lambda\subset V$
with generator $\cL_\Lambda$. 
Then, for all $x \in \Lambda$,  all $\theta \geq 1$, all $q \in (\frac{\theta}{\theta+1},1]$ and all initial configuration $\eta$,
it holds
$$
\bbE_\eta \left( \theta^{\xi^x(\sigma_{t}^{\Lambda})} \right) \le \theta^{\xi^x(\eta)} \,e^{-\lambda t} + \frac{q}{q(\theta+1)-\theta}\quad\quad\forall t\ge 0 ,
$$ 
where $\lambda=\frac{\theta^2-1}{\theta}(q-\frac{\theta}{\theta+1})$.
\end{Proposition}
\begin{proof}
Fix $\theta>1$, $q>0$ and $x\in\Lambda$. To simplify the notation we 
drop the superscript $x$ from $\xi^x$ and set $\xi_t=\xi( \sigma_{t}^{\Lambda})$ in what
follows. Recall that $\sigma_t^\Lambda$ is defined with empty boundary condition so that $\xi_t \le d(x,\Lambda^c)$.
Let
$u(t)=\bbE\_\eta(\theta^{\xi_t})$ and observe that
$$\frac{d}{dt}u(t)=\bbE_\eta(\cL_\Lambda \theta^{\xi_t}).$$

To calculate the expected value above we distinguish two cases: 
{\bf (i)} $\xi_t=0$, {\bf (ii)} $\xi_t\ge 1$.

Case {\bf  (i):} assume that $\xi_t=0$. Then 
\begin{equation}\label{xi0}
(\cL_\Lambda\theta^{\xi_t})\mathds{1}_{\xi_t=0}=\theta^{\xi_t}c_x(\sigma_t^\Lambda)p(\theta-1)\mathds{1}_{\xi_t=0}.
\end{equation}

Case {\bf (ii)}. Define $E(\sigma)=\{y\in V\colon d(x,y)=\xi(\sigma) \mbox{ and } \sigma(y)=0\}$ and
$F(\sigma)=\{y\in V\colon d(y, E)=1\mbox{ and } d(x,y)=\xi(\sigma)-1\}$. Then one argues that $\xi_t$ can increase by 1
only if there is exactly one empty site in the set $E$, and that it can always decrease by 1 by a flip (which is legal by construction) on each site of $F$ (see Figure \ref{fig:xi}). 
\begin{figure}[h] \label{fig:xi}
\psfrag{x}{$x$}
\psfrag{e}{$E$}
\psfrag{f}{$F$}
\includegraphics[width=.70\columnwidth]{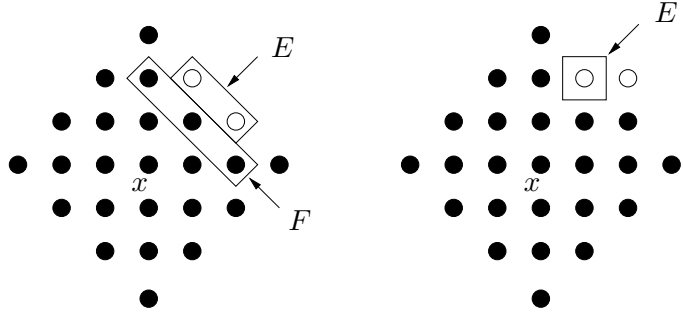}
\caption{On the graph $\cG=\bbZ^2$, two examples of configurations for which $\xi^x=3$.
On the left $\xi^x$ cannot increase since $|E| \ge 2$, it can decrease by a flip (legal thanks to the empty sites in $E$) in any points of $F$.
On the right $\xi^x$ can either increase or decrease.}
\end{figure}

Hence
\begin{align}\label{lteta2}
(\cL_\Lambda\theta^{\xi_t})\mathds{1}_{\xi_t\ge 1}&=\theta^{\xi_t}\left[p(\theta-1)\sum_{y\in E}c_y(\sigma_t^\Lambda)\mathds{1}_{|E|=1}  +q |F|
  (\frac 1\theta
  -1)\right]\mathds{1}_{\xi_t\ge 1}\nonumber\\
  &\le \theta^{\xi_t}[p(\theta-1)-q\frac{\theta-1}{\theta}]+[q\frac{\theta-1}{\theta}-p(\theta-1)]\mathds{1}_{\xi_t=0}
\end{align}
Summing up \eqref{xi0} and \eqref{lteta2} we end up with
$$
\cL_\Lambda\theta^{\xi_t} 
\leq
 \frac {\theta -1}{\theta} \left( \theta^{\xi_t} (p\theta -q) + q \right).
$$
Therefore, since $p=1-q$,
\begin{equation*}
u'(t)\le \frac{\theta-1}{\theta}\left((p\theta-q)u(t)+q \right) = -\lambda u(t) + q \frac{\theta-1}{\theta}
\end{equation*}
and the expected result follows.
\end{proof}


\section{Proof of the Main Theorem}\label{proofmt}

In this section we prove Theorem \ref{main}.

\begin{proof}[Proof of Theorem \ref{main}]

In all the proof $c$ denotes some positive constant depending on all the parameters of the system and that may change from line to line.

Fix $t \ge 2$ and a local function $f$.
Thanks to \eqref{finitespeed} we 
deal with the process in finite volume $\Lambda=B(x,r+100 t)$ where $r \in \bbN$ and $x \in V$ are such that 
$\mathrm{supp}(f) \subset B(x,r)$.

Our aim is to apply Proposition \ref{inizio}. Observe first that
for any positive integer $\ell \leq t$, 
there exists\footnote{One can construct $\Lambda_1,\dots, \Lambda_n$, $x_1,\dots,x_n$ as follows. 
Fix a site $x_o \in \Lambda$ such that $B(x_o, \ell) \subset \Lambda$. 
Then order (arbitrarily) the sites $y_1,y_2,\dots,y_N$ of $\{x \in \Lambda \colon B(x,\ell) \subset \Lambda  \mbox{ and } d(x,x_o)=2 i (\ell+1)-1 \mbox{ for some } i \ge 1 \}$
and perform the following algorithm: set $x_1=x_o$, $i_0=0$, and for $k \ge 1$ 
set $x_{k+1}=y_{i_k}$ with  $i_k:=\inf \{j \ge i_{k-1} + 1: B(y_j,\ell) \cap   (\cup_{i=1}^k B(x_i, \ell)) = \emptyset\}$.
Such a procedure gives the existence of $n$ sites $x_1,\dots,x_n$ such that $B(x_i,\ell) \cap B(x_j, \ell) = \emptyset$, for all $i \neq j$,
$B(x_i,\ell) \subset \Lambda$ for all $i$ 
and any site $y_k \notin A:=\cup_{i=1}^n B(x_i,\ell)$ is at distance at most $2\ell-1$ from $A$. Now attach each connected component
$C$ of $A^c$ to any (arbitrarily chosen) nearest ball $B(x_i,\ell)$, $i \in \{1,\dots,n\}$, with which $C$ is connected, to obtain all the $\Lambda_i$ with the desired properties.} 
a partition  of (connected) sets $\Lambda_1,\dots, \Lambda_n$ of $\Lambda$, and vertices $x_1,\dots,x_n \in V$, such that for any $i$, $B(x_i,\ell) \subset \Lambda_i \subset B(x_i,3\ell)$.

Then, take $\ell = \epsilon [t/\log t]^{1/D}$ if $D>1$ and $\ell=\epsilon t$ if $D=1$ for some $\epsilon>0$ that will be chosen later and 
observe that, with this choice,
$$
M=\max(|\Lambda_1|, \dots, |\Lambda_n|) \le k3^D  \ell^D
$$
(since $\cG$ has polynomial growth).  Furthermore
$$
m=\min(|\Lambda_1|, \dots, |\Lambda_n|) \ge \ell
$$
Since $n \le |\Lambda| \le c t^D$, Equation  \eqref{finitespeed} and Proposition  \ref{inizio} guarantee that
$$
|\mathbb E_{\nu}(f(\sigma_t))| \le 
c \vert| f \vert|_\infty t \sup_{s \in [0,t]} \bbP_\nu(\sigma_s^\Lambda \notin \cA) +
c \vert| f \vert|_\infty
\begin{cases}
e^{-t/c} & \mbox{if } D=1 \\
e^{-[t/(c\log t)]^{1/D}} & \mbox{if } D >1 
\end{cases}
$$
provided $\epsilon$ is small enough.

It remains to study the first term of the latter inequality. We partition each set $\Lambda_i$ into two connected sets $\Lambda_i^+$ and $\Lambda_i^-$ (i.e. $\Lambda_i=\Lambda_i^+ \cup \Lambda_i^-$ and $\Lambda_i^+ \cap \Lambda_i^- = \emptyset$)
such that  for some $x_i^+, x_i^- \in V$, $B(x_i^\pm,\ell/4) \subset \Lambda_i^\pm$ 
(the existence of such vertices are left to the reader).
 The event  $\{\sigma_s^\Lambda \notin \cA\}$ implies that there exists one index $i$ such that at least one of the two halves $\Lambda_i^+, \Lambda_i^-$ 
is completely filled. Assume that it is for example $\Lambda_i^+$, i.e. assume that for any $x \in \Lambda_i^+$, $\sigma_s^\Lambda(x)=1$. 
This implies that 
$\xi^{x_i^+}(\sigma_s^\Lambda) \ge \ell/4$. Hence,  thanks to a union bound, Markov's inequality,  
and Proposition  \ref{prop:xi}, there exists $\theta>1$ such that
\begin{align*}
\bbP_\nu(\sigma_s^L \notin \cA) 
& \le 
2n \bbP_\nu(\xi^{x_i^+}(\sigma_s^\Lambda) \ge \ell/4) \\
& \le 
2n \theta^{-\ell/4}  \bbE_\nu(\theta^{\xi^{x_i^+}(\sigma_s^\Lambda)})\\
& \le 
c n \theta^{-\ell/4} \\
&  \le c
\begin{cases}
e^{-t/c} & \mbox{if } D=1 \\
e^{-[t/(c\log t)]^{1/D}} & \mbox{if } D >1 
\end{cases}
\end{align*}
where we used the definition of $\ell$ and that $n \le |\Lambda| \le c t^D$.
This ends the proof.

\end{proof}


\section{Spectral gap on the ergodic component}\label{ergodicgap}

In this section we estimate the spectral gap of the process FA1f on the ergodic component on 
$\cG=(V,E)$. This has been done in \cite{CMRT,CMRTpraga}. However, we present here an alternative proof, based on the ideas of \cite{MT},
that, on the one hand, gives a somehow more precise bound for very small $q$ and, on the other hand, can be generalized to {\it non cooperative} models different from FA1f  on some ergodic component (not necessarily the largest one).  An example of non cooperative model different from FA1f is the following.
Each vertex $x$ waits an independent mean one exponential time and then, provided that the current configuration $\sigma$ is such that at least two of the sites at distance less or equal to 2 are empty ( $\sum_{y\in \hat \cN_x}(1-\sigma(y))\ge 2$, where $\hat\cN_x=\{y\colon d(x,y)\le2\}$),
the value $\sigma (x)$ is refreshed with a new value in $\{0,1\}$ sampled from a Bernoulli $p$ measure and the whole procedure starts again.
For simplicity we deal with the FA-1f model.

For every $\Lambda\sset V$ finite, define 
\begin{equation} \label{hat}
\hat\Omega_\Lambda=\{\sigma\in \Omega\colon
\sum_{x\in \Lambda}(1-\sigma(x))\ge 1\}
\end{equation}
and 
 $\hat{\mu}_\Lambda(\cdot)=\mu_\Lambda(\cdot\tc \hat\Omega_\Lambda)$. Let $\hat{c}_x(\sigma)=c_{x,\Lambda}^\omega(\sigma)\mathds{1}_{\sigma^x\in\hat\Omega_\Lambda}$ for all $\sigma\in\hat\Omega_\Lambda$, where $\omega$ is the entirely filled configuration, i.e. $\omega(x)=1$ for all $x\in V$. The spectral gap for the dynamics on 
$\hat\Omega_\Lambda$ is defined as

\begin{equation}\label{gap+}
\hat{\rm{gap}}(\Lambda):=\inf_{f\colon f\neq const.}\,\frac{\sum_{x\in\Lambda}\hat{\mu}_\Lambda(\hat{c}_x\var_x(f))}{\var_{\hat{\mu}_\Lambda}(f)}
\end{equation}
where the infimum runs over all non constant functions $f\colon\hat\Omega_\Lambda\to\bbR$,
and for simplicity we set $\var_x (f):= \var_{\mu_{\{x\}}}(f)$.

We now state the result on the spectral gap.

\begin{Theorem} \label{th:gap}
Let $\cG=(V,E)$ be a graph with $(k,D)$-polynomial growth. Then there exists
a positive constant $C=C(k,D)$ such that for any connected set $\Lambda\subset V$ 
$$
\hat{\rm{gap}}(\Lambda)\ge C\,\frac{q^{D+4}}{\log(2/q)^{D+1}}
$$ 
\end{Theorem}

The proof of Theorem \ref{th:gap} is divided in two steps. At first we bound from below the spectral gap of the hat chain in $\Lambda$ by the spectral gap of the FA1f model (not restricted to the ergodic component), on all subsets of $V$ with minimal boundary condition.
Then we study such a spectral gap following the strategy of \cite{MT}.

We need some more notations. Given $A \subset V$, $z \in \partial A$ and  $x \in A$ define
$c_{x,A}^{z}(\sigma) = c_{x,A}^{\omega^{(z)}}(\sigma)$, $\sigma \in \Omega$, where $\omega^{(z)}$ is the entirely filled configuration, except at site $z$ where it is $0$:
$\omega^{(z)}(x)=1$ for all $x \neq z$ and $\omega^{(z)}(z)=0$. The corresponding generator $\cL_A^{\omega^{(z)}}$ will be simply 
denoted by $\cL_A^z$. It corresponds to the FA1f process in $A$ with minimal boundary condition.

The first step in the proof of Theorem \ref{th:gap} is the following result.

\begin{Proposition}\label{gaper}
For every finite connected subsets $\Lambda$ of $V$ with $8p^{\mathrm{diam}(\Lambda)/3}<\frac 12$ it holds
$$
{\rm \hat{gap}}(\Lambda)\ge \,\frac{1}{48}\, \inf_{\genfrac{}{}{0pt}{}{A \subset V, \mathrm{ connected}}{z\in \partial A}} \gap(\cL_A^{z}) .
$$
\end{Proposition}

Observe that, combining \cite[Theorem 6.1]{CMRTpraga} and \cite[Theorem 6.1]{CMRT}
for any set $A$ and any site $z$, we had
$\gap(\cL_A^{z}) \ge c q^{\log_2(1/q)}$
for some universal positive constant $c$. Hence, for the FA1f process, we had the lower bound
$$
{\rm \hat{gap}}(\Lambda) \ge c q^{\log_2(1/q)} .
$$
We present below an alternative strategy (based on \cite{MT}) which can be applied to other non-cooperative models and gives a more accurate bound for the FA1f process  when $q$ is small.

\begin{proof}
Consider a non constant function $f\colon\hat\Omega_\Lambda\to\bbR$ and define 
$\tilde{f}\colon\Omega_\Lambda\to\bbR$  as
$$
\tilde{f}(\sigma)=\begin{cases} f(\sigma) &{\rm if} \, \sigma\in\hat\Omega_\Lambda\\
0&{\rm otherwise}\end{cases}
$$
We divide\footnote{To construct $A$ and $B$ take two points $x,y$ such that $d(x,y)=\ell:=\mathrm{diam}(\Lambda)$ and define $A_0=\{z\in\Lambda\colon d(x,z)\le \ell/3\}$ and $B_0=\{z\in\Lambda\colon d(y,z)\le \ell/3\}$.
Attach to $A_0$ all the connected components of $\Lambda\setminus(A_0\cup B_0)$ connected to $A_0$ to obtain $A$, then attach all the remaining connected components of $\Lambda\setminus(A_0\cup B_0)$  to $B_0$ to obtain $B$.
} $\Lambda$ into two disjoint connected subsets $A$ and $B$ such that their diameter is larger then 
$|\Lambda|/3$.

Thank to Lemma \ref{tec2} below (our hypothesis implies that $\max(1-\mu(c_A), 1-\mu(c_B))<1/16$) 
$$
\var_{\hat\mu_\Lambda}(f)\le 24\, \hat\mu_\Lambda[c_B\var_{\mu_A}(\tilde{f})+
c_A\var_{\mu_B}(\tilde{f})]
$$
where $c_A=\mathds{1}_{\hat \Omega_A}$ and $c_B=\mathds{1}_{\hat \Omega_B}$
and $\hat \Omega_A$ and $\hat \Omega_B$ are defined in \eqref{hat}.

Consider the first term. Define the random variable
$$
\zeta:=\sup_{x \in B}\{d(A,x) \colon \sigma(x)=0\} 
$$
where by convention the supremum of the empty set is $\infty$. 
The function $c_B$ guarantees that $\zeta \in \{1,2,\cdots,\mathrm{diam}(\Lambda)\}$. 
Following the strategy of \cite{CMRT} we have
\begin{align*}
\hat\mu_\Lambda[c_B\var_{\mu_A}(\tilde{f})]&=\frac{1}{\mu_\Lambda(\hat\Omega_\Lambda )}
\sum_{n \ge 1}\mu_\Lambda[\mathds{1}_{\zeta=n}\var_{\mu_A}(\tilde f)]\\
&\le \frac{1}{\mu_\Lambda(\hat\Omega_\Lambda )}
\sum_{n \ge 1}\mu_\Lambda[\mathds{1}_{\zeta=n}\var_{\mu_{A_n}}(\tilde f)]
\end{align*}
where $A_{n}=\{x \in \Lambda \colon d(A,x) \le n-1\}$ and we used the convexity of the variance (which is valid 
since the event $\{\zeta=n\}$ does not depend, by construction, on the value of the configuration $\sigma_{A_n}$ inside $A_n$). 
The indicator
function above $\mathds{1}_{\zeta=n}$ guarantees the presence of a zero on the boundary $\partial A_n$ of the set $A_n$. 
Order (arbitrarily) the points of $\partial A_n$ and call $Z$ the (random) position of the first empty site on $\partial A_n$. 
Then, for all $n \ge 1$, 
\begin{align*}
&\mu_\Lambda[\mathds{1}_{\zeta=n}\var_{\mu_{A_n}}(\tilde f)]  
= 
\sum_{z \in \partial A_n} \mu_\Lambda[\mathds{1}_{\zeta=n}\mathds{1}_{Z=z}\var_{\mu_{A_n}}(\tilde f)] \\
& \qquad \le  
\sum_{z \in \partial A_n}\gap(\cL_{A_n}^z)^{-1}  \sum_{y\in A_n} \mu_\Lambda[\mathds{1}_{\zeta=n}\mathds{1}_{Z=z}
\mu_{A_n}(c_{y,A}^z \var_y(\tilde f))] \\
& \qquad \le  
\gamma \sum_{z \in \partial A_n} \sum_{y\in A_n} \mu_\Lambda[\mathds{1}_{\zeta=n}\mathds{1}_{Z=z} c_{y,A}^z \var_y(\tilde f)]
\end{align*}
where we used the fact that the events $\{\zeta=n\}$ and $\{Z=z\}$ depend only on $\sigma_{A_n}^c$, and
where $\gamma:=\sup \gap(\cL_{A}^z)^{-1}$, the supremum running over all connected subset $A$ of $V$ and all $z \in \partial A$.
Now observe that $\mathds{1}_{\zeta=n}\mathds{1}_{Z=z} c_{y,A}^z \le \mathds{1}_{\zeta=n}\mathds{1}_{Z=z}  \hat c_{y}$ for any $y \in A_{n}$.
Hence, 
\begin{align*}
\hat\mu_\Lambda[c_B\var_{\mu_A}(\tilde{f})]
& \le  
\frac{\gamma}{\mu_\Lambda(\hat\Omega_\Lambda )} \sum_{n \ge 1} \sum_{z \in \partial A_n} \sum_{y\in A_n} \mu_\Lambda[\mathds{1}_{\zeta=n}\mathds{1}_{Z=z}
\hat c_{y} \var_y(\tilde f)] \\
& \le 
\frac{\gamma}{\mu_\Lambda(\hat\Omega_\Lambda )} \sum_{y\in \Lambda} \sum_{n \ge 1} \sum_{z \in \partial A_n}  \mu_\Lambda[\mathds{1}_{\zeta=n}\mathds{1}_{Z=z} \hat c_{y} \var_y(\tilde f)] \\
& 
= \gamma \sum_{y\in \Lambda} \hat \mu_\Lambda[\hat c_{y} \var_y(\tilde f)] 
= \gamma \sum_{y\in \Lambda} \hat \mu_\Lambda[\hat c_{y} \var_y(f)] .
\end{align*}
The same holds for $\hat\mu_\Lambda[c_A\var_{\mu_B}(\tilde{f})]$, leading to the expected result.
\end{proof}

The second step in the proof of Theorem \ref{th:gap} is a careful analysis of $\gap(\cL_A^z)$ for any given connected set $A \subset V$
and $z \in \partial A$. 

\begin{Proposition} \label{prop:cf}
Let $\cG=(V,E)$ be a graph with $(k,D)$-polynomial growth. Then, there exists a universal constant $C=C(k,D)$
such that for any connected set $A \subset V$, and any $z \in \partial A$, it holds
$$
\gap(\cL_A^z) \geq C \frac{q^{D+4}}{\log(2/q)^{D+1}} .
$$
\end{Proposition}

We postpone the proof of  Proposition \ref{prop:cf} to end the proof of Theorem \ref{th:gap}.

\begin{proof}[Proof of Theorem \ref{th:gap}]
The result follows at once combining Proposition \ref{gaper} and Proposition \ref{prop:cf}.
\end{proof}

In order to prove Proposition \ref{prop:cf}, we need a preliminary result on the 
spectral gap of some auxiliary chain,  and to order the points of $A$ in a proper way, depending on $z$.
Let $N :=\max_{x \in A} d(x,z)$, for any $i=1,2,\dots,N$, we define
$$
A_i := \{x \in A \colon d(x,z)=i\}=\{x_1^{(i)},\dots,x_{n_i}^{(i)}\}
$$
where $x_1^{(i)},\dots,x_{n_i}^{(i)}$ is any chosen order.
Then we say that for any $x, y \in A$,  $x \le y$  if either $d(x,z) > d(y,z)$ or
$d(x,z)=d(y,z)$ and $x$ comes before $y$ in the above ordering. 
Then, we set $A_x = \{y \in A \colon y \ge x \}$ and $\tilde A_x =A_x \setminus \{x\}$.

\begin{Lemma} \label{lemma}
Fix a connected set $A \subset V$, and  $z \in \partial A$. 
For any $x \in A$ and $\sigma \in \Omega$, let $E_x \subset \Omega_{\tilde A_x}$,
$\Delta_x=\mathrm{supp}(E_x)$ 
and $\tilde c_x(\sigma)=\mathds{1}_{E_x}(\sigma_{\tilde A_x})$. 
Assume that 
$$
\sup_{x \in A} \mu(1-\tilde c_x) \sup_{x \in A} |\{y\in A\colon\Delta_y\cup\{y\}\ni x\}|< \frac{1}{4} .
$$
Then, for any $f : \Omega_A \to \bbR$ it holds
$$
\var_{\mu_A} (f) \le 4 \sum_{x \in A} \mu_A ( \tilde c_x \var_x (f)) .
$$
\end{Lemma}

\begin{proof}
We follow \cite{MT}. 
In all the proof, to simplify the notations, we set $\var_B=\var_{\mu_B}$, for any $B$.
First, we claim that
\begin{equation} \label{var}
\var_{A}(f) = \sum_{x\in A} \mu_A (\var_{A_x} (\mu_{\tilde A_x}(f))) .
\end{equation}
Take $x=x_{n_N}^{(N)}$, by factorization of the variance, we have
$$
\var_{A}(f) = \mu_A (\var_{\tilde A_{x}}(f)) + \var_A( \mu_{\tilde A_{x}}(f)) .
$$
The claim then follows by iterating this procedure, removing one site at a time, in the order defined
above.

We analyze one term in the sum of \eqref{var} and assume, without loss of generality, that
$\mu_{A_x}(f)=0$. We write $\mu_{\tilde A_x}(f)=\mu_{\tilde A_x}(\tilde c_x f)
+ \mu_{\tilde A_x}((1-\tilde c_x)f)$ so that
\begin{equation} \label{var2}
\mu_A [\var_{A_x} (\mu_{\tilde A_x}(f))]
\le 
2 \mu_A [\var_{A_x} (\mu_{\tilde A_x}(\tilde c_x f)) ]+ 
2 \mu_A [\var_{A_x} (\mu_{\tilde A_x}((1-\tilde c_x)f) )].
\end{equation}
Observe that, by convexity of the variance and since $\tilde c_x$ does not depend on $x$,
the first term of the latter can be bounded as
$$
\mu_A [\var_{A_x} (\mu_{\tilde A_x}(\tilde c_x f)) ] 
= 
\mu_A [\var_{x} (\mu_{\tilde A_x}(\tilde c_x f)) ]
\le
\mu_A [\tilde c_x \var_{x} ( f) ] .
$$
Now we focus on the second term of \eqref{var2}. Note that
$
\mu_{\tilde A_x}[(1-\tilde c_x)f) ]
= \mu_{\tilde A_x}[(1-\tilde c_x)\mu_{\tilde A_x\setminus \Delta_x}(f)) ]$.
Set $\delta:=\sup_{x \in A} \mu(1-\tilde c_x)$. Hence, bounding the variance by the second moment and using Cauchy-Schwarz inequality, we get
\begin{align*}
 \var_{A_x} (\mu_{\tilde A_x}((1-\tilde c_x)f) ) & \le
  \var_{A_x} \left(\mu_{\tilde A_x}[(1-\tilde c_x)\mu_{\tilde A_x\setminus \Delta_x}(f)]\right) \\
 & \le
  \mu_{A_x}\left( \mu_{\tilde A_x}[(1-\tilde c_x)\mu_{\tilde A_x\setminus \Delta_x}(f)]^2\right)   \\
 &\le   \delta \left( \var_{A_x} (\mu_{\tilde A_x\setminus \Delta_x}(f)) \right)
\end{align*}
From all the previous computations (and using \eqref{var})  we deduce that
\begin{align*}
\var_A(f) 
& \le 2 \sum_{x \in A} \mu_A ( \tilde c_x \var_x (f)) 
+ 2 \delta \sum_{x \in A} \mu_{A} \left( \var_{ A_x} (\mu_{\tilde A_x\setminus \Delta_x}(f)) \right) .
\end{align*}
Hence if one proves that
\begin{equation}\label{var3}
\sum_{x \in A} \mu_{A} \left( \var_{ A_x} (\mu_{\tilde A_x\setminus \Delta_x}(f)) \right) \le \sup_{y \in A} |\{x\in A\colon\Delta_x\cup\{x\}\ni y\}| \var_A(f)
\end{equation}
the result follows. We now prove \eqref{var3}. Using \eqref{var}, we have
$$
\var_{A_x} (g) 
= 
\sum_{y \in  A_x} \mu_{ A_x} \left( \var_{A_y} ( \mu_{\tilde A_y}(g) )\right)
= 
\sum_{y \in \Delta_x\cup\{x\}} \mu_{ A_x} \left( \var_{A_y} ( \mu_{\tilde A_y}(g) )\right)
$$
where $g=\mu_{\tilde A_x\setminus \Delta_x}(f)$ and we used that $\mathrm{supp}(g) \subset \Delta_x$.
It follows that
$$
\mu_A\left( \var_{  A_x} (g)  \right) \!\!
\sum_{y \in \Delta_x\cup\{x\}} \!\! \mu_{A} \left( \var_{A_y} ( \mu_{\tilde A_y}(g) )\right)
\le
\!\! \sum_{y \in \Delta_x\cup\{x\}} \!\! \mu_{A} \left( \var_{A_y} ( \mu_{\tilde A_y}(f) )\right)
$$
since, by Cauchy-Schwarz,
\begin{align*}
\mu_{A} \left( \var_{A_y} ( \mu_{\tilde A_y}(g) )\right) 
& = 
\mu_{A} \left(  \left[ \mu_{\tilde A_x\setminus \Delta_x} \left(\mu_{\tilde A_y}(f) - \mu_{A_y}(f)\right) \right]^2 \right) \\
& \le
 \mu_{A} \left( \var_{A_y} ( \mu_{\tilde A_y}(f) )\right) .
\end{align*}
This ends the proof.
\end{proof}

\begin{proof}[Proof of Proposition \ref{prop:cf}]
Our aim is to apply Lemma \ref{lemma}. Let us define the events $E_x$, for $x \in A$.
Fix an integer $\ell$ that will be chosen later and set $n=\ell \wedge d(x,z)$.
Let  $(x_1,x_2,\dots,x_n)$ be an arbitrarily chosen ordered collection
satisfying  
 $d(x_i,x_{i+1})=1$, $d(x_i,x)=i$ and $d(x_i,z)=d(x,z)-i$  for $i=0,\dots,n$, with the convention that $x_0=x$, and set $E_x=\{\sigma \in \Omega \colon \sum_{i=1}^n (1-\sigma(x_i)) \ge 1 \}$, i.e. $E_x$ is the event that at least one of the site of $\Delta_x=\{x_1,x_2,\dots,x_n\}$ is empty. Note that by construction $\Delta_x \subset A \cup\{z\}$ and is connected. Moreover
 for any $x$ such that $d(x,z)\le \ell$, $E_x=\Omega$ so that $\tilde c_x \equiv 1$.
 Since  $|\Delta_x| \leq k\ell^D$ for any $x \in A$, the assumption of Lemma \ref{lemma} reads 
 $$
 p^\ell(1+k \ell^D) < 1/4
 $$
 which is satisfied if one chooses $\ell = \frac{c}{q} \log \frac{2}{q}$ with $c=c(k,D)$ large enough. Hence for any $f : \Omega_A \to \bbR$ it holds
$$
\var_{\mu_A} (f) \le 4 \sum_{x \in A} \mu_A ( \tilde c_x \var_x (f)) .
$$
and we are left with the analysis of  each term $\mu_A ( \tilde c_x \var_x (f))$ for which we use a path argument.
Fix $x \in A$ and the collection $(x_1,x_2,\dots,x_n)$ introduced above.
Given a configuration $\sigma$ such that $\tilde c_x (\sigma)=1$, denote by $\xi$ the (random) distance
between $x$ and the first empty site in the collection $(x_1,x_2,\dots,x_n)$: i.e. $\xi(\sigma)=\inf \{i \colon \sigma(x_i)=0\}$.
Then we write
\begin{align*}
\mu_A ( \tilde c_x \var_x (f)) 
& = \sum_{i=1}^n \mu_A ( \tilde c_x \mathds{1}_{\xi=i} \var_x (f)) \\
&  pq \sum_{i=1}^n \sum_{\sigma : \xi(\sigma)=i} \mu_A (\sigma) (f(\sigma^x) - f(\sigma))^2 
\end{align*}
where the sum is understood to run over all $\sigma$ such that $\tilde c_x(\sigma)=1$ (and $\xi(\sigma)=i$).

Fix $i \in \{1,\dots,n\}$. For any $\sigma \in \Omega$ such that $\xi(\sigma)=i$, we construct a path of configurations 
$\gamma_x(\sigma)=(\sigma_0=\sigma,\sigma_1,\sigma_2,\dots,\sigma_{4i-5}=\sigma^x)$
from $\sigma$ to $\sigma^x$,  of length $4i -5 \le 4\ell$. The idea behind the construction is to bring an empty site from $x_i$, step by step, toward $x_1$, make the flip in $x$ and going back, keeping track of the initial configuration $\sigma$.
For any $j$, $\sigma_{j+1}$ can be obtained from $\sigma_j$ by a legal flip for the FA1f process. Furthermore
$\sigma_j$ differs from $\sigma$ on at most three sites (possibly counting $x$). 
More precisely, define $T_k(\sigma):=\sigma^{x_k}$ for any $k$ and $\sigma$, and
$$
\sigma_j = 
\begin{cases}
T_{i-k-1}(\sigma) & \mbox{if } j=2k+1, \mbox{ and } k=0,1,\dots,i-2 \\
T_{i-k} \circ T_{i-k-1}(\sigma) & \mbox{if } j=2k, \mbox{ and } k=1,\dots,i-2 \\
T_{1}(\sigma^x) & \mbox{if } j=2i-2 \\
T_{k-i+2} \circ T_{k-i+3}(\sigma^x) & \mbox{if } j=2k+1, \mbox{ and } k=i-1,\dots,2i-4 \\
T_{k-i+2}(\sigma^x) & \mbox{if } j=2k, \mbox{ and } k=i,\dots,2i-3 .
\end{cases}
$$
See Figure \ref{aiuto} for a graphical illustration of such a path. 

\begin{figure}[h] \label{aiuto}
\psfrag{x0}{\footnotesize $\;\;x$}
\psfrag{x1}{\footnotesize $ x_1$}
\psfrag{x2}{\footnotesize $\,x_2$}
\psfrag{x3}{\footnotesize $ x_3$}
\psfrag{x4}{\footnotesize $\!\!\!\!\!x_4=x_i$}
\psfrag{s0}{$\sigma_0=\sigma$}
\psfrag{s1}{$\sigma_1$}
\psfrag{s2}{$\sigma_2$}
\psfrag{s3}{$\sigma_3$}
\psfrag{s4}{$\sigma_4$}
\psfrag{s5}{$\sigma_5$}
\psfrag{s6}{$\sigma_6=T_1(\sigma^x)$}
\psfrag{s7}{$\sigma_7$}
\psfrag{s8}{$\sigma_8$}
\psfrag{s9}{$\sigma_9$}
\psfrag{s10}{$\sigma_{10}$}
\psfrag{s11}{$\sigma_{11}=\sigma^x$}
\includegraphics[width=.70\columnwidth]{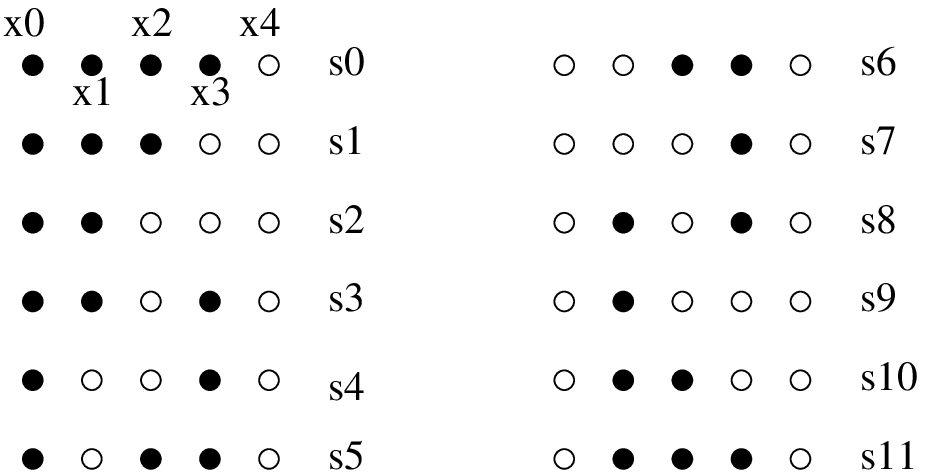}
\caption{Illustration of the path from $\sigma$ to $\sigma^x$ for a 
configuration $\sigma$ satisfying
$\xi(\sigma)=x_4$. Here $i=4$ and the length of the path is $4i-5=11$.}
\end{figure}

Denote by $\Gamma_x(\sigma)=\{\sigma_0,\sigma_1,\dots,\sigma_{4i-6}\}$ (i.e. the configurations of the path $\gamma_x(\sigma)$ except the last one $\sigma^x$).
For any $\eta =\sigma_j \in \Gamma_x(\sigma)$, $j \ge 1$, let $y=y(x,\eta) \in \{x,x_1,x_2,\dots,x_\ell\}$ be such that $\eta=\sigma_{j-1}^{y}$.  Then, by Cauchy-Schwarz inequality,
\begin{align*}
(f(\sigma^x)-f(\sigma) )^2
& = \left( \sum_{\eta \in \Gamma_x(\sigma)} (f(\eta^y) - f(\eta)) \right)^2 
\le 4\ell \sum_{\eta \in \Gamma_x(\sigma)} \left( f(\eta^y) - f(\eta) \right)^2 \\
& \le 
\frac{4\ell}{p q} \sum_{\eta \in \Gamma_x(\sigma)} c_y(\eta) \var_y (f) (\eta).
\end{align*}
Hence,
\begin{align*}
\mu_A ( \tilde c_x \var_x (f))  
& \le
4\ell K \sum_{\eta } \mu_A(\eta) c_y(\eta) \var_y (f) 
\end{align*}
where
$$
K= 
\sup_{\eta \in \Omega, x \in A} \left\{ \sum_{\sigma} \sum_{i=1}^\ell \frac{\mu_A(\sigma)}{\mu_A(\eta)} \mathds{1}_{\xi(\sigma)=i}
\mathds{1}_{\Gamma_x(\sigma) \ni \eta} \right\}   
\le   \frac{8}{q^3}  .
$$
Indeed $\mu_A(\sigma)/\mu_A(\eta) \le \frac{p^2}{q^2} \max(\frac{p}{q}, \frac{q}{p})$ since  any $\eta \in \Gamma_x(\sigma)$ has at most two extra empty sites with respect to $\sigma$  and differs from $\sigma$ in at most three sites, and we used a computing argument.

Recall that $y=y(x,\eta)$. It follows from the latter that
\begin{align*}
\var_{\mu_A} (f) 
& \le
\frac{128 \ell}{q^3} \sum_{x \in A} \sum_{\eta } \mu_A(\eta) c_y(\eta) \var_y (f) \\
& \le
\frac{128\ell}{q^3} K'\sum_{u \in A} \sum_{\eta } \mu_A(\eta) c_u(\eta) \var_u (f)  
\end{align*}
where 
$$
K'= \sup_{\eta} \sum_{x \in A} \mathds{1}_{y(x,\eta)=u} \le \sup_{u \in A} |B(u,\ell)| .
$$
The result follows since the graph has polynomial growth.
\end{proof}

In Proposition \ref{gaper} we used the following lemma.
\begin{Lemma}\label{tec2}
Take  $\Lambda, A, B \sset V$ such that $\Lambda=A\cup B$ and $A\cap B=\emptyset$. 
Define
$c_A=\mathds{1}_{\hat \Omega_A}$ and 
$c_B=\mathds{1}_{\hat \Omega_B}$ where 
$\hat \Omega_A$ and $\hat \Omega_B$ are defined in \eqref{hat}.
Assume that $\max(1-\mu(c_A), 1-\mu(c_B))<1/16$.
Then, for all $f\colon \hat \Omega_\Lambda \to \bbR$ with $\hat \mu_\Lambda(f)=0$ it holds
$$
\var_{\hat \mu_{\Lambda}}(f)\le 24 \hat \mu_{\Lambda}[c_B\var_{\mu_A}(\tilde{f})+
c_A\var_{\mu_B}(\tilde{f})]
$$
where $\tilde{f}\colon\Omega_\Lambda\to\bbR$ is defined as
$$
\tilde{f}(\sigma)=\begin{cases} f(\sigma) &{\rm if} \, \sigma\in\hat \Omega_\Lambda\\
0&{\rm otherwise}\end{cases}
$$
\end{Lemma}

\begin{proof}
Recalling the definition of the variance we have
\begin{align*}
\var_{\hat \mu_\Lambda}(f)&=\inf_{m\in\bbR} \hat \mu_\Lambda(|f-m|^2)\\
&\le\frac{1}{\mu_\Lambda(\hat \Omega_\Lambda)}\,\inf_{m\in\bbR} \mu_\Lambda((f\mathds{1}_{\hat \Omega_\Lambda}-m)^2)\\
&=\frac{1}{\mu_\Lambda(\hat \Omega_\Lambda)}\var_{\mu_\Lambda}(\tilde{f}).
\end{align*}
Observe now that, by construction,  $\mu_\Lambda(\tilde{f})=0$ and $(1-c_A)(1-c_B)\tilde{f}=0$ so that we can apply  Lemma \ref{tec1} below and obtain
$$
\var_{\hat \mu_\Lambda}(f)\le \frac{24}{\mu_\Lambda(\hat \Omega_\Lambda)}\,\mu_\Lambda[c_B\var_{\mu_A}(\tilde{f})+
c_A\var_{\mu_B}(\tilde{f})]
$$
and the result follows.
\end{proof}

The next Lemma might be heuristically seen as a result on the spectral gap of some constrained blocks dynamics (see \cite{CMRT}). 
Such a bound can be of independent interest.

\begin{Lemma}\label{tec1}
Let $\Lambda=A\cup B$ with $A,\, B\subset V$ satisfying $A\cap B=\emptyset$.
Define  $\mu_A$ and $\mu_B$ two probability measures on $\{0,1\}^A$ and $\{0,1\}^B$ respectively, and $\mu=\mu_A\otimes\mu_B$. Take $c_A,\, c_B\colon\{0,1\}^\Lambda\to [0,1]$ with support in $A$ and $B$ respectively. For any function $g$ on $\{0,1\}^\Lambda$ 
such that $(1-c_A)(1-c_B)g=0$
it holds
\begin{align*}
\var_\mu(g) & \le 12\mu[c_B^2\var_{\mu_A}(g)+c_A^2\var_{\mu_B}(g)] \\
& \quad +8\max(1-\mu(c_A), 1-\mu(c_B))\var_\mu(g) .
\end{align*}
\end{Lemma}

\begin{proof}
Fix $g$ on $\{0,1\}^\Lambda$  such that $(1-c_A)(1-c_B)g=0$ and assume without loss of generality that $\mu(g)=0$.
First we write 
\begin{align*}
g&=c_B(g-\mu_A(g))+(1-c_B)c_A(g-\mu_B(g))+(1-c_B)c_A\mu_B(g)\\
&\quad - (1-c_B)c_A\mu_A(g)+(1-c_B)(1-c_A)(g-\mu_A(g))+\mu_A(g)\\
&=c_B(g-\mu_A(g))+(1-c_B)c_A(g-\mu_B(g))+(1-c_B)c_A\mu_B(g)+c_B\mu_A(g)
\end{align*}
where we used the first hypothesis on $g$, $(1-c_A)(1-c_B)g=0$, and we arranged the terms.
Therefore since we assumed $\mu(g)=0$ and $c_A,\,c_B\in[0,1]$
\begin{align*}
\var_\mu(g)=\mu(g^2)&\le 4\mu(c_B^2(g-\mu_A(g))^2)+4\mu(c_A^2(g-\mu_B(g))^2)\\
&\quad +4\mu(\mu_B(g)^2)+4\mu(\mu_A(g)^2)\\
&=4\mu[c_B^2\var_{\mu_A}(g)+c_A^2\var_{\mu_B}(g)] \\
&\quad+4\mu(\mu_B(g)^2)+4\mu(\mu_A(g)^2).
\end{align*}
We now treat the fourth term in the latter inequality.
\begin{align*}
[\mu_A(g)]^2&=[\mu_A(g)-\mu(g)]^2=[\mu_A(g-\mu_B(g))]^2\\
&=[\mu_A(c_A[g-\mu_B(g)])+\mu_A([1-c_A][g-\mu_B(g)])]^2\\
&\le 2\mu_A(c_A^2)\mu_A(c_A^2[g-\mu_B(g)]^2)+2\mu_A((1-c_A)^2)\mu_A([g-\mu_B(g)]^2)
\end{align*}
If we average with respect to $\mu$ we have
$$
\mu(\mu_A(c_A^2)\mu_A(c_A^2[g-\mu_B(g)]^2))=\mu_A(c_A^2)\mu(c_A^2\var_{\mu_B}(g))
$$
and, using Cauchy-Schwarz inequality and $x^2\le x$ for $x\in[0,1]$,
\begin{align*}
\mu( \mu_A((1-c_A)^2)\mu_A([g-\mu_B(g)]^2))&=\mu_A((1-c_A)^2)\mu([g-\mu_B(g)]^2)\\
&\le (1-\mu(c_A))\var_\mu(g),
\end{align*}
so that
$$
\mu(\mu_A(g)^2)\le 2 \mu_A(c_A^2)\mu(c_A^2\var_{\mu_B}(g)) + 2 (1-\mu(c_A))\var_\mu(g).
$$
An analogous calculation for $\mu(\mu_B(g)^2)$ allows to conclude the proof.
\end{proof}


\bibliography{BibFA1f.bib}
\bibliographystyle{amsplain}
\underline{\hspace{5cm}}

\end{document}